\documentclass[12pt]{amsart}
\usepackage{amssymb,latexsym}
\usepackage{enumerate}
\usepackage{verbatim}
\usepackage{amsfonts}
\usepackage{amsmath}
\usepackage{fancyhdr}
\usepackage{amsmath,amsxtra,amssymb,amsthm,latexsym,amscd}
\usepackage{array}
\makeatletter
\@namedef{subjclassname@2010}{%
  \textup{2010} Mathematics Subject Classification}
\makeatother

\newlength{\defbaselineskip}
\newcommand{\setlinespacing}[1]%
           {\setlength{\baselineskip}{#1 \defbaselineskip}}

\newtheorem{thm}{Theorem}[section]

\newtheorem{lem}[thm]{Lemma}

\newcommand{\ds}{\displaystyle}
\newcommand{\tx}{\textstyle}

\theoremstyle{definition}
\newtheorem{df}[thm]{Definition}
\newtheorem{rem}[thm]{Remark}

\newtheorem{Corol}[thm]{Corollary}

\numberwithin{equation}{section}

\newcommand{\R}{\mathbb{R}}
\newcommand{\N}{\mathbb{N}}
\newcommand{\B}{\mathbf{B}}
\newcommand{\rank}{{\textrm{rank}}}

\begin{document}
{

\baselineskip=17pt



 \title[ a numerical approach to some basic theorems in singularity theory]
{a numerical approach to some basic theorems in singularity theory}

\author[Ta L\^e Loi and  Phan Phien]{Ta L\^e Loi and  Phan Phien}

\address{department of mathematics, university of dalat, dalat, vietnam }
\email{loitl@dlu.edu.vn}

\address{nhatrang college of education\\  1 nguyen chanh, nhatrang, vietnam}
\email{phieens@yahoo.com}

\date{}

\begin{abstract}
In this paper, we give the explicit bounds for the data of objects
involved in some basic theorems of Singularity theory:  the Inverse,
Implicit and Rank Theorems for Lipschitz mappings, Splitting Lemma
and Morse Lemma, the density and openness of Morse functions. We
expect that the results will make Singularities more applicable and
will be useful for Numerical Analysis and some fields of computing.

\end{abstract}
 \maketitle

\section{introduction}
\renewcommand{\thefootnote}{}

To make Singularity Theory more applicable it is important to make
its basic results `quantitative'. This direction of the theory is
proposed by Y. Yomdin in \cite{Y1} where he proves the quantitative
Morse-Sard theorem by giving the notion of near-critical values of
differentiable mappings and estimated these sets by the metric
entropy. For the discussion on this direction and its developments we refer the readers to \cite{Y2} and \cite{Y-C} and the references therein. We are
interested in the numerical approach of this direction. In this
paper, we give the quantitative versions of some basic theorems of
Singularity theory: the Inverse, Implicit and Rank Theorems for
Lipschitz mappings, Splitting Lemma and Morse Lemma, the density and
openness of Morse functions. The explicit bounds for the data of the
objects involved are estimated via the input data (e.g. $C^k$-norms,
radii of the balls, ... ).
 The main tools that we use are some familiar methods of Singularities of
differentiable mappings (see \cite{A-G-V}, \cite{B-L}, \cite{G-G},
or \cite{Ma}),  the quantitative forms of the Inverse and Implicit
mappings theorems, and Morse-Sard theorem (see \cite{C1}, \cite{Pa}, \cite{P},  Theorems
\ref{Inverse} and \ref{Implicit} in this paper, \cite{Y1} and \cite{Y-C}). In our results, the
estimates of the first order derivatives and the radii of the domains
of the mappings involved are quite sharp. Since we use Lemmas
\ref{lemE} and \ref{lemEI}, the estimates of the higher order
derivatives of the mappings involved are  explicit but rather big.
We expect that the results will make Singularities more applicable
and will be useful for Numerical Analysis and some fields of
computing.

The plan of our paper  is as follows: In Section 2 we recall some
definitions and give the estimates of $C^k$-norms of compositions
and inverses. In Section 3 we consider the quantitative versions of
the Inverse, Implicit and Rank Theorems for Lipschitz mappings. In
Section 4 we give a quantitative form of diagonalization of
matrix-valued mappings by upper triangular matrices, the
quantitative versions of Splitting Lemma and Morse Lemma, and
applications to the density and openness of Morse functions on a
ball.

\section{preliminaries}
We give here some definitions, notations and results that will be used later.

Let $\mathbf{M}_{m\times n}$ denote the vector space of real $m \times n$ matrices,
\begin{itemize}
\item[] $\|x\| =(|x_1|^2 + \cdots + |x_n|^2)^{\frac{1}{2}}$, ~where $x\in \mathbb{R}^n$,\\
 $\mathbf{B}_r^n(x_0)$ denotes the ball of radius $r$, centered at $x_0$ in $ \mathbb{R}^n$,
 $\mathbf{B}_r^n=\mathbf{B}_r^n(0)$, and $\mathbf{B}^n=\mathbf{B}_1^n$,
\item[] $\|A\| = \max_{\|x\|=1}\|Ax\|, ~$where $A \in \mathbf{M}_{m\times
n}$, or $A$ is a linear mapping.
\item[] $\mathcal{B}_{m\times n}$ denotes the unit ball in $\mathbf{M}_{m\times n}$,
\item[] $\textrm{Sym}(n)$ denotes the space of real symmetric  $n\times n$-matrices,
\item[] $\Delta(n)$ denotes the vector space of all upper triangular $n \times n$-matrices.
\end{itemize}
\begin{df}[see \cite{C2}] \label{dn2} 
Let $ f: \mathbb{R}^n  \rightarrow  \mathbb{R}^m $ be a Lipschitz mapping in a  neighborhood
$U$  of $ x_0 $ in $ \mathbb{R}^n$, i.e. there exists a constant $ K>0 $ such that
\[ \|f(x) - f(y)\|\leq K\|x-y\|, \text{ for all } x,y\in U.
\]
Then we denote  $L(f)\leq K$.
\\
By Rademacher's theorem (see \cite{F}), a Lipschitz mapping on a subset
$U$ of $\R^n$ is differentiable almost everywhere. The Jacobian matrix of the partial
derivatives of $f$ at $x$, when it exists, is denoted by $Jf(x)$.
 The \textbf{generalized Jacobian} of $f$
at $x_0$, denoted by $\partial f(x_0)$, is the convex hull of all
matrices $M$ of the form
\[M = \lim_{i\rightarrow \infty}Jf(x_i),\]
where $(x_i)$ converges to $x_0$ and $f$ is differentiable at $x_i$ for each $i$.\\
For $p\leq \min(m,n)$, we denote
\[\partial_{p\times p} f(x_0)=\{M_1\in\mathbf{M}_{p\times p}: \textrm{there exists } M=\left(\begin{array}{cc} M_1 & M_2 \\ M_3 & M_4
\end{array}\right)\in \partial f(x_0)\}.\]
$\partial f(x_0)$ is said to be of \textbf{maximal rank} if every $M$ in $\partial f(x_0)$ has the maximal rank.\\
$\partial f(x_0)$ is said to be of \textbf{rank  $p$} if every $M$
in $\partial f(x_0)$ has rank $p$.
\end{df}
\begin{df}
Let $f: U \rightarrow \mathbb{R}^m$ be a differentiable mapping of
class $C^k$, $k\geq 1$,  on an open subset $U$ of $\mathbb{R}^n$.
Then the $C^k$-norm of $f$ is defined by
\[\|f\|_{C^k} =  \max_{1\leq p\leq k}\sup_{x\in U}\|D^pf(x)\|.\]
\end{df}
In the next sections we have to estimate the $C^k$-norm of
compositions and inverses, to this aim  we prepare the following two
lemmas.
\begin{lem}\label{lemE}
 Let $f:U\to V$ and $g:V\to\R^p$ be differentiable mapping of class $C^k$, $k\geq
 1$, on open subsets $U\subset\R^n, V\subset\R^m$. Then
\[\ \|g\circ f\|_{C^k}\leq (1^k+2^k+\cdots+k^k)\|g\|_{C^k}\max(\|f\|_{C^k},\|f\|_{C^k}^k) .\]
\end{lem}
\noindent
In this paper, we  denote
\[E(K_f,K_g, k)=(1^k+2^k+\cdots+k^k)K_g\max(K_f,K_f^k).\]
\begin{proof}
By the Higher Order Chain Rule (see \cite{A-M-R}),
 for $p\leq k$, we get the following estimation
 \[ \begin{array}{lll}
 \|D^p(g\circ f)(x)\|&\leq
 \ds\sum_{i=1}^p\ds\sum_{j_1+\cdots+j_i=p}\frac{p!}{j_1!\cdots j_i!}
 \|D^ig(f(x))\|\|D^{j_1}f(x)\|\cdots \|D^{j_i}f(x)\|\\
 &\leq \ds\sum_{i=1}^p\ds\sum_{j_1+\cdots+j_i=p}\frac{p!}{j_1!\cdots j_i!}\|g\|_{C^k}\|f\|_{C^p}^i\\
 &\leq \ds\sum_{i=1}^p i^p\ \|g\|_{C^k}\max(\|f\|_{C^k}, \|f\|_{C^k}^p).
 \end{array}
 \]
 So $\|g\circ f\|_{C^k}\leq (1^k+2^k+\cdots+k^k)\|g\|_{C^k}\max(\|f\|_{C^k},\|f\|_{C^k}^k).$
 \end{proof}
\begin{lem}\label{lemEI}
  Let $\varphi:U\to V$  be a $C^k$ diffeomorphism between open subsets $U, V$ of $\R^n$, $k\geq 1$.
  Then we have the estimation
\[\ \|\varphi^{-1}\|_{C^k}\leq EI(\|\varphi\|_{C^k}, \|D\varphi^{-1}\|,k),\]
where $EI$ is constructed by the following recurrent method:
\\
 Let
 \  $ M_0= E(\|\varphi\|_{C^k}, \ds\max_{0\leq p\leq k-1}p!\|D\varphi^{-1}\|^{p+1}, k-1),
  \  M_1=\|D\varphi^{-1}\|,$\\
 and\ $ \ M_p=E(M_{p-1}, M_0, p-1), \textrm{ for }  p=2,\cdots, k.$
 Then \ $EI(K,L,k)= M_k$.
\end{lem}
\begin{proof} Let
\[
\begin{array}{rccl}
\textrm{Inv}:&\textrm{Gl}(n)&\rightarrow& \textrm{Gl}(n)\\
&M &\mapsto&\textrm{Inv} (M) = M^{-1}.
\end{array}
\]
Since $ D\varphi^{-1}= \textrm{Inv}\circ
 D\varphi\circ\varphi^{-1}$, using Lemma \ref{lemE},  we have the recurrent inequalities
\[
 \|\varphi^{-1}\|_{C^p}=\max(\|D\varphi^{-1}\|, \|D\varphi^{-1}\|_{C^{p-1}})\leq
 E(\|\varphi^{-1}\|_{C^{p-1}},\|\textrm{Inv}\circ D\varphi\|_{C^{p-1}}, p-1),\text{for } p\geq 2.
 \]
 First, we estimate  $\|\textrm{Inv}\circ D\varphi\|_{C^{k-1}}$.
 From $D^p\textrm{Inv}(M)(\delta M)=p!(-1)^p(M^{-1}\delta
M)^pM^{-1}$, we get $
\|D^p\textrm{Inv}(M)\| \leq p!\|M^{-1}\|^{p+1}$.
 Therefore, using the notations $K=\|\varphi\|_{C^k}, L=\|D\varphi^{-1}\|$, we have
 \[
 \|D^{p}\textrm{Inv}(D\varphi(x))\|  \leq p!\|(D\varphi(x))^{-1}\|^{p+1}\leq p!L^{p+1}.\]
 Using Lemma \ref{lemE}, we get
 \[\|\textrm{Inv}\circ D\varphi\|_{C^{k-1}}\leq E(K,\max_{1\leq p\leq k-1}{p!}{L^{p+1}},k-1)=M_0.\]
 Let
 \[ M_1= L , M_p=E(M_{p-1}, M_0, p-1), p=2, \cdots, k.\]
From $\|\varphi^{-1}\|_{C^1}=\|D\varphi^{-1}\|\leq M_1$, using Lemma
\ref{lemE} and recurrence, we have
 \[\|\varphi^{-1}\|_{C^k}\leq M_k= EI(K,L,k).\]
\end{proof}
\begin{df}[see \cite{G-L}]
 Let $L: \mathbb{R}^n \rightarrow \mathbb{R}^m$ be a linear mapping. Then there exist $\sigma_1(L) \geq \ldots \geq \sigma_r(L)>0$, where $r = \rank L$, so that $L(\mathbf{B}^n)$ is an $r$-dimensional ellipsoid of semi-axes $\sigma_1(L) \geq \ldots \geq \sigma_r(L)$. Set $\sigma_0(L) = 1$ and  $\sigma_{r+1}(L) = \ldots = \sigma_m(L) = 0$, when $r < m$.
We call $\sigma_0(L), \ldots, \sigma_m(L)$ the\textbf{ singular values} of $L$.
\end{df}
\begin{rem}\label{nx221} \textit{ }
\begin{itemize}
\item[(i)] $\sigma_1(L)=\|L\|$,\ $\sigma_{m}(L)= \min_{\|x\|=1}\|Lx\|$.
\item[(ii)]  If $\lambda\in\mathbb{R}$ is an  eigenvalue of $L$,
then \ $ \ds\sigma_{m}(L) \leq |\lambda| \leq \sigma_{1}(L).$
\end{itemize}
\end{rem}
\begin{df}
A $C^k$ function  $f:U\to\R$ on an open subset $U$ of $\R^n$, $k\geq
2$,  is called {\bf Morse } if for every critical point $x$ of $f$,
i.e. $Df(x)=0$, the Hessian $Hf(x)$ is nondegenerate, i.e.
$\sigma_n(Hf(x))>0$.
\end{df}
\section{the inverse, implicit  and rank theorems for lipschitz mappings }
In this section, we present  quantitative forms of the Inverse,
Implicit and Rank Theorems for Lipschitz mappings, and give some
explicit bounds in the smooth case.

\begin{thm}[Inverse Mapping Theorem]\label{Inverse}
Let $f: \mathbb{R}^n \rightarrow\mathbb{R}^n$ be a Lipschitz mapping. Suppose that $~\partial f(x_0)$ is of maximal rank. Set
\[\delta = \frac{1}{2}\inf_{M_0\in \partial f(x_0)}\frac{1}{\|M_0^{-1}\|},\]
and $r$ be chosen so that $L(f)\leq K$  and
$$\partial f(x) \subset \partial f(x_0) + \delta \mathcal{B}_{n\times n},~~ \textrm{when}~~x \in \mathbf{B}_r^n(x_0).$$
Then $f:\mathbf{B}^n_{\rho_1}(x_0)\to f(\mathbf{B}^n_{\rho_1}(x_0))$ is a homeomorphism, where $\rho_1=\frac{r\delta}{2K}$,
$ f(\mathbf{B}^n_{\rho_1}(x_0))$ contains $\mathbf{B}^n_{\rho_2}(f(x_0))$, where $\rho_2=\frac{r\delta}{2}$, and $L(f^{-1})\leq \ds\frac{1}{\delta}$
on $\mathbf{B}^n_{\rho_2}(f(x_0)).$
\end{thm}
\begin{proof} See \cite{C1} and \cite{P}.
\end{proof}
\begin{rem}\label{remInverse}
 We make some comments on the pair $(\delta,r)$ of the
theorem that are often used latter.\\
Let $\Sigma=\{ A\in \textrm{M}_{n\times n}: \det A=0 \}$. Then by
the Eckart-Young equality (see \cite{G-L}), we have
\[ \ds\frac 1{\|M^{-1}\|}= d(M, \Sigma), \textrm{ for every } M
\textrm{ in } \textrm{M}_{n\times n} \setminus \Sigma.\]
 Hence,
\[ \delta=\ds\frac 12\inf_{M_0\in \partial
f(x_0)}\frac{1}{\|M_0^{-1}\|}=\ds\frac 12 d(\partial f(x_0),
\Sigma).
\]
In words, $\delta$ is half the distance from the generalized
Jacobian of $f$ at $x_0$ to the singular locus  $\Sigma$. Note that
if $\delta'\leq \delta$, then the theorem is also true when $\delta$
is replaced by $\delta'$.
\\
By the upper semicontinuous property of the generalized Jacobian
(see \cite{C2}), for every $\delta>0$ there exists $r>0$, such that
$$\partial f(x) \subset \partial f(x_0) + \delta \mathcal{B}_{n\times n},~~ \textrm{when}~~x \in \mathbf{B}_r^n(x_0).$$
So the quantity $r$ reflects the rate of variation of the generalized
Jacobian of $f$ in a neighborhood of $x_0$. If $r'\leq r$, then the
theorem is also true when
$r$ is replaced by $r'$.\\
\end{rem}
Using Lemma \ref{lemEI}, we have the following corollary.
\begin{Corol}\label{corolInverse}
 With the assumptions and notations of Theorem \ref{Inverse},
and in addition  $f$ is a $C^k$ mapping, $k\geq 2$, and $\|f\|_{C^k}
\leq K$. Then we can choose
\[\delta =\ds\frac 1{2\|Df(0)^{-1}\|}, \textrm{ and  } \  r=\ds\frac{\delta}K.\]
Moreover $f^{-1}$ is also in class $C^k$, $\|D f^{-1}\|\leq \frac
1{\delta}$, and
\[ \|f^{-1}\|_{C^k} \leq EI(K,\frac 1{\delta},k).\]
\end{Corol}
\begin{rem}If $F: U\times V \rightarrow \mathbb{R}^n$ be a Lipschitz mapping in a neighborhood of $(x_0, y_0)$ in  $\mathbb{R}^m \times \mathbb{R}^n$, then the generalized Jacobian of $F$ at $(x_0, y_0)$ satisfies
\[\partial F(x_0, y_0) \subset \left\{\left(
                                  \begin{array}{cc}
                                    M_1 & M_2 \\
                                  \end{array}
                                \right): M_1 \in \partial_1 F(x_0, y_0), M_2 \in \partial_2 F(x_0, y_0)\right\},\]
where $\partial_1 F(x_0, y_0)$ and $\partial_2 F(x_0, y_0)$ are the
generalized Jacobians of $F(\cdot, y_0): U \rightarrow \mathbb{R}^n$
and $F(x_0, \cdot): V \rightarrow \mathbb{R}^n$ at $(x_0, y_0)$,
respectively.
\end{rem}
\begin{thm}[Implicit Function Theorem]\label{Implicit}
Let $F: \R^m\times\R^n \rightarrow \mathbb{R}^n$ be a Lipschitz mapping in a neighborhood of $(x_0, y_0)$. Suppose that $\partial_2F(x_0, y_0)$ is of maximal rank  and  $F(x_0, y_0) = 0$. Set
\[\delta = \frac{1}{2} \inf_{M_2 \in \partial_2 F(x_0, y_0)}\frac{1}{(1+(1+K)^2\| M_2^{-1}\|^2)^\frac{1}{2}},\]
and $r$ be chosen so that  $L(F)\leq K$ and
 $$ \partial F (x, y) \subset \partial F(x_0, y_0) + \delta \mathcal{B}_{n\times(m+n)}, ~\textrm{when}~ (x, y) \in  \mathbf{B}_r^{m+n}((x_0, y_0)).$$
Then there exists a Lipschitz mapping $g: \mathbf{B}_{\rho}(x_0)
\rightarrow \R^n$, where $\rho= \frac{r  \delta}{2(K+1)}$, and
$L(g)\leq \frac K{\delta}$,   such that
\[ g(x_0) = y_0 , \textsl{ and } \ F(x, g (x)) = 0,
\textsl{ when  } x \in \mathbf{B}_{\rho}(x_0).\]
 \end{thm}
 \begin{proof} (c.f. \cite{Pa},\cite{P}).\\
Set $f(x, y) = (x, F (x, y))$.
 Then $f$ is Lipschitz with $L(f)\leq K+1$ and
\[\partial f(x_0, y_0) \subset \left\{\left(\begin{array}{cc}
                                       I_m & 0 \\
                                       M_1 & M_2
                                     \end{array}
\right): \ \ M_1 \in \partial_1 F(x_0, y_0), \ \ M_2 \in \partial_2 F(x_0, y_0) \right\}.\]
Since $\partial_2 F(x_0, y_0)$ is of maximal rank, $\partial f(x_0, y_0)$ is of maximal rank.
\\
For $M = \left(\begin{array}{cc}
                                       I_m & 0 \\
                                       M_1 & M_2
                                     \end{array}
\right) \in \partial f(x_0, y_0)$, we have
$M^{-1}=\left(\begin{array}{cc} I_m & 0 \\   -M_2^{-1}M_1 & M_2^{-1} \\
             \end{array}
           \right).
$\\
 Therefore,
\[ \|M^{-1}\|
=\ds\sup_{\|x\|^2+\|y\|^2 = 1}\left(\|x\|^2 + \|M_2^{-1}M_1x -
M_2^{-1}y\|^2 \right)^\frac{1}{2} \leq  \left(1 +
(K+1)^2\|M_2^{-1}\|^2\right)^\frac{1}{2},
\]
and hence
\[\frac{1}{\|M^{-1}\|} \geq \frac{1}{\left(1 + (K+1)^2\|M_2^{-1}\|^2\right)^\frac{1}{2}}.\]
Since
\[ \Delta = \frac{1}{2}\inf_{M\in \partial f(x_0, y_0)}\frac{1}{\|M^{-1}\|}\geq\delta,
\]
by  the supposition, we have
\[ \partial f (x, y) \subset \partial f(x_0, y_0) + \Delta \mathcal{B}_{n\times n}, \textrm{when }
(x, y) \in  \mathbf{B}_r^{m+n}((x_0, y_0)).\]
 By Theorem \ref{Inverse}, $f$ is locally invertible, with $f^{-1}(x, z) = (x,
h(x, z))$,  $(x,z)\in \mathbf{B}^{n+m}_{\rho}(x_0,0)$, where
$\rho=\ds\frac{r\delta}{2(K+1)}.$ Let $g(x) = h(x, 0),  x\in
\mathbf{B}^m_{\rho}(x_0)$.
Then $g$ is Lipschitz, $g(x_0)=y_0$ and $F(x, g (x))=0$.\\
Moreover, when $F$ is differentiable at $(x,g(x))$, we have
\[\begin{array}{rcl}
\|Dg(x)\| & = & \|-\left(\frac{\partial F}{\partial y} \right)^{-1} \frac{\partial F}{\partial x}(x,g(x))\|\\
& \leq & \ds\sup_{M_2 \in \partial_2F(x, y), \ (x, y) \in \mathbf{B}_r^{m+n}(x_0, y_0)} \|M_2^{-1}\| K\\
& \leq & \ds\sup_{M_2 \in \partial_2F(x_0, y_0) + \delta\mathcal{B}_{n \times n}} \|M_2^{-1}\| K
 \leq \frac{1}{\delta} K.
\end{array}
\]
So \ $L(g)\leq \frac K{\delta}$.
 \end{proof}

\begin{Corol}\label{corolImplicit}
 With the assumptions and notations of Theorem \ref{Implicit},
and in addition  $F$ is a $C^k$ mapping, $k\geq 2$, and $\|F\|_{C^k} \leq K$. Then we can choose
\[\delta =\ds\frac 1{2(1+(1+K)^2\|\frac{\partial F}{\partial y}(0,0)^{-1}\|^2)^{\frac 12}}, \textrm{ and  } \  r=\ds\frac{\delta}K.\]
Moreover $g$ is also in class $C^k$, $\|Dg\|\leq \frac K{\delta}$,
and
\[ \|g\|_{C^k} \leq C(K,k)=2^{k-1}EI(K,\tx\frac K{\delta},k-1) K.\]
\end{Corol}

\begin{proof} From $F(x, g(x)) = 0$,  we have $Dg(x) =
-\left(\frac{\partial F}{\partial y} \right)^{-1}\frac{\partial
F}{\partial x}(x,g(x)).$ \\
By Theorem \ref{Implicit}, $\|g\|_{C^1}=\|Dg\|\leq \frac K{\delta}.$\\
Applying the Higher Order Leibnitz Rule (see \cite{A-M-R}), we have
\[\begin{array}{lll}
\|g\|_{C^k} &=\max(\|Dg\|, \|Dg\|_{C^{k-1}})\\
&\leq \ds\sum_{i=0}^{k-1}\left(\begin{array}{c} k-1 \\
i\end{array}\right)
 \|\left(\tx\frac{\partial F}{\partial y}\right)^{-1}\|_{C^{k-1}}\|\tx\frac{\partial F}{\partial x}\|_{C^{k-1}}\\
&\leq 2^{k-1}\|\left(\frac{\partial F}{\partial y}
\right)^{-1}\|_{C^{k-1}}\ K.
\end{array}\]
To estimate $\|\left(\frac{\partial F}{\partial
y}\right)^{-1}\|_{C^{k-1}} $, we use Lemma \ref{lemEI} to get
\[
\|\left(\frac{\partial F}{\partial y}\right)^{-1}\|_{C^{k-1}}
\leq
 EI(\|\tx\frac{\partial F}{\partial y}\|_{C^{k-1}},\|D\left(\tx\frac{\partial F}{\partial y}\right)^{-1}\|,k-1)
 \leq  EI(K,\frac K{\delta}, k-1).
 \]
 From this estimation,  we get
\[\|g\|_{C^k} \leq 2^{k-1}EI(K,\tx\frac K{\delta},k-1) K.\]
\end{proof}

\begin{thm}[Rank Theorem]\label{Rank}
Let $f: \R^n \rightarrow \R^m$ be a Lipschitz mapping in a neighborhood $U$ of  $x_0\in \R^n$ with $L(f)\leq K$. Suppose that $\partial f(x)$ is of rank $p$ for all $x \in U$ and
that $\partial_{p\times p} f(x_0)$ is of rank $p$. Set
\[\delta = \frac{1}{2}\inf_{M_1\in \partial_{p\times p} f(x_0)}\frac{1}{(1+(1+K)^2\|M_1^{-1}\|^2)^{\frac 12}},\]
and $r$ be chosen so that $\mathbf{B}_r^n(x_0)\subset U$ and
\[\partial f(x) \subset \partial f(x_0) + \delta \mathcal{B}_{n \times n}, \ \textrm{for all} \ x \in \mathbf{B}_r^n(x_0).\]
Then there exist  homeomorphisms
\[\varphi: \mathbf{B}^n_{\rho_1}(x_0) \rightarrow \varphi(\mathbf{B}^n_{\rho_1}(x_0)) , \textrm{ and  \ }
\psi: \mathbf{B}^m_{\rho_2}(f(x_0)) \to \psi(\mathbf{B}^m_{\rho_2}(f(x_0)), \]
where $\rho_1=\frac{r\delta}{2(K+1)}, \rho_2=\frac{r\delta}2, L(\varphi)\leq \frac 1{\delta},$ and $L(\psi)\leq 1 +\frac{K}{\delta}$,\\
such that $ \varphi(\mathbf{B}^n_{\rho_1}(x_0))\supset
\mathbf{B}^n_{\rho_2}(\varphi(x_0))$,
$\varphi^{-1}:\mathbf{B}^n_{\rho_2}(\varphi(x_0)) \to \R^m$,
$L(\varphi^{-1})\leq \frac 1{K+1}$, and
\[\psi\circ f \circ \varphi^{-1}(z_1, \ldots, z_n) = (z_1, \ldots, z_p, 0, \ldots, 0).\]

\end{thm}
\begin{proof} Without loss of generality we can assume $x_0 = 0, f(0) = 0$.\\
Let  \ $\varphi: (\R^n,0 )\to (\R^n,0 )$ , be defined by
\[\ \varphi(x_1, \ldots, x_n) = (f_1(x), \ldots, f_p(x), x_{p+1}, \ldots, x_n).\]
Then
$L(\varphi)\leq K+1$, and
each  $M \in \partial \varphi(0)$ has the form
\[M=\left(
\begin{array}{cc}
M_1 & M_2 \\
 0     & I
\end{array}
\right), \ \textrm{where }\ M_1 \in \partial_{p\times p}f(0).\]
Hence $\partial \varphi(0)$ is of maximal rank. By an estimation in
the proof of Theorem \ref{Implicit}, we have
\[ \|M^{-1}\| \leq  (1+(1+K)^2\|M_1^{-1}\|^2)^{\frac 12}.\]
So
$ \delta_1=\ds\frac{1}{2}\ds\inf_{M \in\partial \varphi(0)}\frac{1}{\|M^{-1}\|}\geq \delta= \frac{1}{2}\inf_{M_1\in \partial_{p\times p} f(x_0)}\frac{1}{(1+(1+K)^2\|M_1^{-1}\|^2)^{\frac 12}}$.
From this inequality and the assumption, we have
\[\partial \varphi(x) \subset \partial \varphi(0) + \delta_1 \mathcal{B}_{n \times n}, \ \textrm{for all} \ x \in \mathbf{B}_r^n.\]
 Hence we can apply Theorem \ref{Inverse}, to conclude  that $\varphi :\mathbf{B}_{\rho_1}\to  \varphi(\mathbf{B}_{\rho_1}) $ is a homeomorphism,
where $\rho_1=\frac{r\delta}{2(K+1)},$  $
\varphi(\mathbf{B}_{\rho_1})\supset \mathbf{B}_{\rho_2}$, where
$\rho_2=\frac{r\delta}2$, and $L(\varphi^{-1})\leq
\frac{1}{\delta}$.
\\
Let \[ g = f\circ \varphi^{-1}: \mathbf{B}_{\rho_2} \to \R^m.\]
Then $L(g) \leq L(f)L(\varphi^{-1}) \leq \frac{K}{\delta}$, and $g(z_1,\ldots, z_n)= (z_1, \ldots, z_p, g_{p+1}(z), \ldots, g_m(z)).$\
Therefore,  each $ M' \in \partial g(z)$ is of  the form
 \[M' = \left(
 \begin{array}{cc}
 I& 0\\
 M_1 & M_2
 \end{array}
 \right).\]
Since  $\partial g(z)$ is of rank $p$ when $z\in \mathbf{B}_{\rho_2}$,
  \begin{equation}\label{ct171}
 M_2 =0.
 \end{equation}
Let \ $\psi: \mathbf{B}_{\rho_2}\rightarrow  \R^m $, be given by
 \[
 \left(
 \begin{array}{c}
 y_1\\
 \vdots\\
 y_p\\
 y_{p+1}\\
 \vdots\\
 y_m
 \end{array}
 \right) \mapsto   \left(
  \begin{array}{c}
 y_1\\
  \vdots\\
  y_p\\
  y_{p+1}-g_{p+1}(y_1, \ldots, y_p, 0, \ldots, 0)\\
  \vdots\\
  y_m-g_{m}(y_1, \ldots, y_p, 0, \ldots, 0)
  \end{array}
  \right)
 \]
Then $L(\psi)\leq 1+ L(g) \leq 1+ \frac{K}{\delta}$, and each $M'' \in \partial\psi(y)$ is of the form
 \[M'' = \left(
 \begin{array}{ccc}
 I& 0\\
 ? & I
 \end{array}
  \right) \]
By  Theorem \ref{Inverse}, $\psi$ is locally invertible. It is easy
to see that $\psi$ is injective. So $\psi$  is a homeomorphism from
$\mathbf{B}_{\rho_2}$ onto its image.
\\
Because of (\ref{ct171}),  $g_{j}(z_1,\cdots, z_m) - g_{j}(z_1, \ldots, z_p, 0, \ldots, 0)=0$, for every  $z\in \mathbf{B}_{\rho_2}$ and $j >p$ . Therefore
$\psi \circ f \circ \varphi^{-1}=\psi\circ g $ is represented by
\[(z_1, \ldots, z_n) \mapsto (z_1, \ldots, z_p, 0, \ldots, 0).\]
\end{proof}
\begin{Corol}\label{corolRank}
With the assumptions and notations of Theorem \ref{Rank}, and in
addition $f$ is of class $C^k$, $k\geq 2$, and $\|f\|_{C^k} \leq K$.
Then we can choose
\[ \delta=\ds\frac 1{2(1+(1+K)^2\|J_{p\times p}f(x_0)\|^2)^{\frac 12}} , \textrm{ and } \ r=\ds\frac K{\delta}.\]
Moreover  $\varphi$ and $\psi$ are  of  class $C^k$,
$\|\varphi\|_{C^k} \leq K+1$, $\|D\psi\| \leq \frac K{\delta}+1$,
and
$$ \|\psi\|_{C^k} \leq C(K, \delta,k)=E(EI(K+1,\frac
1{\delta}, k), K,k).$$
\end{Corol}
\begin{proof}
It suffices to estimate  $\|\varphi\|_{C^k}$ and $\|\psi\|_{C^k}$.\\
By the definition of $\varphi$, we have $\|\varphi\|_{C^k} \leq
\|f\|_{C^k} +1 \leq  K+1$.
\\
From $\|D\varphi^{-1}\|\leq \frac{1}{\delta}$, using Lemma
\ref{lemEI}, we have $\|\varphi^{-1}\|_{C^k}\leq EI(K+1, \frac
1{\delta}, k)$. \\
 From this inequality, using Lemma \ref{lemE}, we get
\[\|g\|_{C^k} =\|f\circ \varphi^{-1}\|_{C^k} \leq E( EI(K+1, \frac
1{\delta}, k), K,k).\] Finally, by the definition of  $\psi$,
$\|\psi\|_{C^k}\leq \|g\|_{C^k}+1 $, and hence
 \[ \|D\psi\|\leq \|Dg\|+1\leq \tx\frac K{\delta}+ 1, \text{ and }
  \  \|\psi\|_{C^k} \leq  E(EI(K +1,\frac
1{\delta}, k), K,k) +1.\]
\end{proof}
\section{splitting lemma -  morse functions.}
In this section we give the quantitative versions of Splitting
Lemma, Morse Lemma, the density and openness of Morse functions on a
ball.

To prove the Splitting Lemma, we prepare the following lemma, which
gives a quantitative form of diagonalization of matrix-valued
mappings by upper triangular matrices.
\begin{lem}\label{Diag}
 Let
$\overline{B}: U\to  \textrm{Sym}(n),$ be a $C^k$ matrix-valued
mappings, $k\geq 1$, on a open neighborhood  $U$ of   $\ 0$ in $\R^n$. Suppose that $\|\overline{B}\|_{C^k} \leq \bar{K}$, and
\[\overline{B}(0) = D_0 = \textrm{diag}(\pm 1, \ldots, \pm 1).\]
Let  
$\delta=\ds\frac{1}{4(\bar{K}+1)(\bar{K}+2)(1+
(\bar{K}+2)^2)n(n+1))}.$ Then there exists  a $C^k$ mapping
\[\mathcal{Q}: U_{\delta} \rightarrow \Delta(n), \textrm{ where }  U_{\delta} = U \cap \mathbf{B}_\delta^n,\]
such that \ $\mathcal{Q}(0) = I_n, \ \overline{B}(x) =
{^t\mathcal{Q}}(x) D_0 \mathcal{Q}(x), $ and
\[\|D\mathcal{Q}(x)\|\leq (\bar{K}+1)(1+(\bar{K}+2)n(n+1) ) , \|\mathcal{Q}\|_{C^k} \leq \overline{C}(\bar{K},n,k)
  ,\]
where \
 $\overline{C}(\bar{K},n,k)=2^{k-1}(\bar{K}+1)EI(\bar{K}+1,(\bar{K}+1)(1+(\bar{K}+2)n(n+1)),k-1).$
\end{lem}
\begin{proof}
Let $s=\ds\frac{n(n+1)}2=\dim \Delta(n)=\dim \textrm{Sym}(n)$. \\
Consider
$ F: U \times \Delta(n) \rightarrow \textrm{Sym}(n),
F(x, Q) = \overline{B}(x) - {^tQ}D_0Q.$
\\
Then
$ F \in C^k,
F(0, I_n) = 0 \in \mathbf{M}_{n\times n}$, and
\[ \frac{\partial F}{\partial Q}(0, I_n): \Delta(n) \rightarrow \textrm{Sym}(n),
 \frac{\partial F}{\partial Q}(0, I_n)(H)= {-^tH}D_0 - D_0H.
\]
Denote  $H = (h_{ij}), G = (g_{ij}) = -({^tH}D_0 + D_0H)$. Then
\[\left\lbrace
\begin{array}{lcl}
h_{ij} &=& - g_{ij}, i < j, \\
h_{ij} & =& \pm \frac{1}{2} g_{ii}.
\end{array}.
\right.\]
Thus $\frac{\partial F}{\partial Q}(0, I_n) \ \textrm{is invertible}, $ and
\[
\begin{array}{ll}
\| \frac{\partial F}{\partial Q}(0, I_n)\| = \ds\sup_{\|H\|=1}\|{^tHD_0}-D_0H\|  \leq \ds\sup_{\|H\|=1}2\|H\|\|D_0\| = 2,\\
\|(\frac{\partial F}{\partial Q}(0, I_n))^{-1}\| \leq \|(\frac{\partial F}{\partial Q}(0, I_n))^{-1}\|_{F} \leq \sqrt{s} = \sqrt{\frac{n(n+1)}{2}}.
\end{array}
\]
We are using  Implicit Function Theorem, so we estimate some numbers related to $F$.
From
\begin{equation}\label{ctDF}
 \  F(x+\Delta x, Q + H) - F(x, Q) =
\overline{B}(x+\Delta x) - \overline{B}(x) - {^tH}D_0Q - {^tQ}D_0H-
{^tH Q}D_0 H,
\end{equation}
we have
\[\frac{\partial F}{\partial x}(x, Q) = D\overline{B}(x), \
\frac{\partial F}{\partial Q}(x, Q)(H) = {-^tH}\overline{B}(x)Q -
{^tQ}\overline{B}(x)H. \] For $0<r\leq
\frac{\sqrt{1+2\bar{K}}}{\sqrt{2}\bar{K}}$, when $\|(x, Q)\| \leq r
$,  we have
\[\begin{array}{rcl}
\|DF(x, Q)\| & \leq &( \|D\overline{B}(x)\|^2 + \ds\sup_{\|H\|=1}\|{^tH}\overline{B}(x)Q + {^tQ}\overline{B}(x)H\|^2)^{\frac{1}{2}} \\
&\leq&\left( \bar{K}^2+2\|\overline{B}(x)\|^2 \|Q\|^2
\right)^{\frac{1}{2}} \leq \bar{K}\sqrt{1+2r^2} \leq \bar{K} +1.
\end{array}
\]
So $L(F)\leq \bar{K} +1 $ \ on  $\mathbf{B}_r^{n+ s}$. To apply
Theorem \ref{Implicit} to $F$, with
\[\delta_1 = \frac {1}{2 \left( 1+(1+(\bar{K}+1)^2 s) \right)^{\frac{1}{2}}},\]
we chose
 $r = \min(\frac{\sqrt{1+2\bar{K}}}{\sqrt{2}\bar{K}}, \frac{\delta_1}{\bar{K}+1}) =  \frac{\delta_1}{\bar{K}+1}$
 to have
\[\|DF(x, Q) - DF(0, I_n)\|  < (\bar{K}+1)r = \delta_1, \textrm{ when } \|(x,Q)\|<r.\]
According to Theorem \ref{Implicit} and Corollary
\ref{corolImplicit}, there exists a $C^k$ mapping
\[\mathcal{Q}: U_{\delta} \rightarrow \Delta(n) \ \textrm{in class} \ C^k, \ U_{\delta} =  U \cap \mathbf{B}_\delta^n, \]
where $\delta =
\ds\frac{r\delta_1}{2(\bar{K}+2)}=\ds\frac{1}{8(\bar{K}+1)(\bar{K}+2)(1+
(\bar{K}+2)^2)s}, $ such that
\[ \mathcal{Q}(0)= I_n, \ F(x, \mathcal{Q}(x)) = \overline{B}(x) -{^t\mathcal{Q}(x)}D_0\mathcal{Q}(x) = 0,
\textrm{ and }  \|D\mathcal{Q}(x)\| \leq
\frac{\bar{K}+1}{\delta_1}.\] Moreover, from (\ref{ctDF}), we have
\[\ds
\begin{array}{l}
\frac{\partial^p F}{\partial x^p}(x, Q) = D^p\overline{B}(x), \\
\frac{\partial F}{\partial Q}(x, Q)(\Delta x, \Delta Q) = -{^t\Delta Q}D_0 Q - {^tQ}D_0\Delta Q, \textrm{and hence }
\|\frac{\partial F}{\partial Q}\|\leq 2\|Q\|\leq 2, \\
\frac{\partial^2 F}{\partial Q^2}(x, Q)(\Delta x, \Delta Q) = -{^t\Delta Q}D_0 \Delta Q,
\textrm{ and hence } \|\frac{\partial^2 F}{\partial Q^2}(x, Q)\| \leq 1,\\
\frac{\partial^{|\alpha|+|\beta|} F}{\partial x^{\alpha} \partial
Q^{\beta}}(x, Q) = 0, \textrm{ when } \alpha\in\N^n, \beta\in \N^s,
\alpha\neq 0\neq \beta \textrm{ or }|\beta|\geq 3.
\end{array}
\]
So $\|F\|_{C^k} \leq \bar{K}+2.$ Using  Corollary
\ref{corolImplicit}, we get  \[\|\mathcal{Q}\|_{C^k} \leq
 \overline{C}(\bar{K}, n,k)=2^{k-1}(\bar{K}+2)EI(\bar{K}+2,\frac{\bar{K}+2}{\delta_1},k-1).\]
\end{proof}
%
Applying the above lemma and Implicit Function Theorem,  we can  get
a quantitative form of Splitting Lemma.
\begin{thm}[Splitting Lemma]\label{Split}
Let $f: U \rightarrow \R$ be a $C^k$ function on a neighborhood
$U$ of $x_0$ in $\R^n$, $k \geq 3$. Suppose that $\|f\|_{C^k}\leq
K$, and
\[Df(x_0) = 0, \  \rank  D^2f(x_0) = p. \]
 Let \
 \[\delta=\ds\frac{\sigma_p^{5/2}}{32(K+1)^{9/2}}
 \min(1,\ds\frac{2\sigma_p^2}{3(p^2+p+1)}, \ds\frac{\sigma_p^{3/2}}{2(p^2+p+1)} ), \textrm{ where } \sigma_p=
\sigma_p(D^2f(x_0)).\]
 Then there exists  a $C^{k-1}$
diffeomorphism
\[\varphi: \mathbf{B}_\delta^n \rightarrow \varphi(\mathbf{B}_\delta^n), \textrm{ with } \
 \|D\varphi\|\leq \frac{32(K+1)^5}{\sigma_p^{5/2}}, \]
such that
\[f \circ \varphi(x, y) = f(x_0)+ \sum_{i=1}^p \pm x_i^2 + \alpha(y), \ x=(x_1, \ldots, x_p) \in \R^p, y\in \R^{n-p},\]
where $\alpha$ is of class $C^{k}$  and $\alpha(0), D\alpha(0), D^2\alpha(0)$ vanish.\\
Moreover, there exists a constant $M(K, \sigma_p, k)>0$, such that
 \[\|\varphi\|_{C^{k-1}} \leq M(K, \sigma_p, k).\]
\end{thm}

\begin{proof}
We can assume  $x_0=0, f(x_0)=0$, and
 can choose the coordinate system $(x, y) \in \R^p \times \R^q, p + q = n, x = (x_1, \ldots, x_p), y = (y_1, \ldots, y_q)$ so that
\[A = H_1f(0, 0) = \left(\frac{\partial^2f}{\partial x_i \partial x_j}(0, 0) \right)_{1 \leq i, j \leq p}\]
is of rank $p$.
 Note that if $\sigma_1 \geq \cdots \geq \sigma_p > 0$ are the singular values of $A$, then
 $\frac{1}{\sigma_p} \geq \cdots \geq \frac{1}{\sigma_1} > 0$ are the  singular values of $A^{-1}$ and
\[\frac{1}{\sigma_1} \leq \|A^{-1}\| \leq  \frac{1}{\sigma_p}.\]
\textbf{Step 1.} Consider the equation:
 \ $ \frac{\partial f}{\partial x}(x, y) = 0.$
\\
We have  $\frac{\partial f}{\partial x}(0, 0) = 0$. To apply the Implicit
Function Theorem \ref{Implicit} to $\frac{\partial f}{\partial x}$ ,
we will determine some numbers. Let
\[\delta '= \frac{1}{2(1+(K+1)^2\frac{1}{\sigma_p^2})^{\frac{1}{2}}},
\textrm{ and } r'=\ds\frac{\delta'}K.\]
Then, for $\|(x,y)\|<r'$,
\[\|D(\frac{\partial f}{\partial x})(x, y) - D(\frac{\partial f}{\partial x})(0, 0) \| < Kr'=\delta'.\]
So we can apply Theorem \ref{Implicit} and its
corollary \ref{corolImplicit}, to have a $C^{k-1}$ mapping
\[g: \mathbf{B}_{\delta_1}^q \rightarrow \R^p,\]
 where
 \[\delta_1 = \frac{r' \delta'}{2(K+1)}=\ds\frac 1{8K(K+1)(1+(K+1)^2\frac 1{\sigma_p^2})},\]
such that
\[g(0) = 0, \ \ \frac{\partial f}{\partial x}(g(y), y) = 0, \]
 and
\begin{equation}\label{ctDg}
 \|Dg\| \leq \frac{K}{\delta_1}, \
  \ \|g\|_{C^{k-1}} \leq2^{k-2}EI(K,\frac K{\delta_1}, k-2)K =M_1(K,\sigma_p).
\end{equation}
Let $\alpha(y)=f(g(y), y)$. Then $\alpha(0)=0, D\alpha(0)=0$, and
\[D\alpha(y)=\frac{\partial f}{\partial x}(g(y), y)Dg(y)+\frac{\partial f}{\partial y}(g(y),
 y)=\frac{\partial f}{\partial y}(g(y),y).\]
 Since the mapping on the right side is of class $C^{k-1}$,
 $D\alpha$ is of class $C^{k-1}$,  and hence $\alpha$ is of class
 $C^k$. Moreover,  $\frac{\partial f}{\partial x}(g(y), y)\equiv 0$ and
$\frac{\partial^2 f}{\partial y^2}(0, 0) = 0$ imply $D^2\alpha(0)=0$.
\\
Let \ $f_1(x, y) =f(x, y) - \alpha(y),   (x, y) \in \mathbf{B}_{\delta_1}^n.$
We have
\[ 
f_1(g(y), y) =0, \
\ds\frac{\partial f_1}{\partial x}(g(y), y) =0.
\]
\textbf{Step 2.} Let  $\varphi_1(x, y) = (x + g(y), y), (x,y)\in
\mathbf{B}_{\delta_1}^n.$ Then $\varphi_1$ is a $C^{k-1}$
diffeomorphism from  $\mathbf{B}_{\delta_1}^n$ to its image, and
from (\ref{ctDg}), we get
\begin{equation}\label{ctDphi1}
\|D\varphi_1\|\leq 1+\|Dg\|\leq 1+\frac K{\delta_1}, \ \
\|\varphi_1\|_{C^{k-1}} \leq 1 + M_1(K,\sigma_p)= M_2(K,\sigma_p).
\end{equation}
Let   $f_2=f_1\circ \varphi_1$.  Then $f_2$ is of class $C^{k-1}$,
and
\[f_2(0, y) = 0, \  \frac{\partial f_2}{\partial x}(0, y) = 0.\]
Note that
\[\frac{\partial f_2}{\partial x}(x, y) = \frac{\partial f_1}{\partial x}(x + g(y), y),
\textrm{ and } \
 \frac{\partial^2f_2}{\partial x^2}={^t\frac{\partial\varphi_1}{\partial x}}\frac{\partial^2f_1}{\partial
x^2}\frac{\partial\varphi_1}{\partial x}=
\frac{\partial^2f_1}{\partial x^2}= \frac{\partial^2f}{\partial
x^2}.\]
\textbf{Step 3.}
Let $Q_0 \in \textrm{Gl}(p)$ be the linear transformation so that
\[{^tQ_0}AQ_0 = D_0 = (\pm 1, \ldots, \pm 1).\]
Moreover, choose $Q_0 = SU$, where $U$ is an orthogonal matrix and
$S$ is a diagonal matrix, so that
\[\|Q_0\|^2 = \frac{1}{\sigma_p}, \|Q_0^{-1}\|^2=\sigma_1=\|A\| \leq K.\]
Let $B: \mathbf{B}_{\delta_1}^n \rightarrow \textrm{Sym}(p)$,
defined by
\[B(x, y) = \left(b_{ij}(x, y) \right)_{1\leq i, j \leq p}, \textrm{where }
b_{ij}(x, y) = \int_0^1\int_0^1\frac{\partial^2 f_2}{\partial x_i
\partial x_j}(stx, y) ds dt.\]
 Then $B$ is of class $C^{k-1}$, and
\[f_2(x, y) = {^tx}B(x,y)x , \ \textrm{and} \ B(0, 0) = A = H_1f(0, 0).\]
Set
\[\overline{B}(x, y) = {^tQ_0}B(x, y)Q_0.\]
 Then  $\overline{B}: \mathbf{B}_{\delta_1}^n \rightarrow\textrm{Sym}(p) \in C^{k-1}$, $\overline{B}(0,0) = D_0$,
and \
$\|D\overline{B}\|\leq \bar{K}=\frac K{\sigma_p}$.
According to  Lemma \ref{Diag}, there exists a $C^{k-1}$ mapping
\[\mathcal{Q}: \mathbf{B}_{\delta_2}^p \rightarrow \Delta(p), \]
where
$\delta_2=\min(\delta_1,\ds\frac 1{4(\bar{K}+1)(\bar{K}+2)(1+(\bar{K}+2)^2p(p+1))})$,\\
such that
\[\mathcal{Q}(0) = I_p,
\overline{B}(x, y) = {^t\mathcal{Q}}(x)D_0\mathcal{Q}(x),
\|D\mathcal{Q}\|\leq (\bar{K}+1)(1+(\bar{K}+2)^2p(p+1)),\]
 and
 \begin{equation}\label{ctDQ}
 \|\mathcal{Q}\|_{C^{k-1}}\leq \overline{C}(\bar{K},p,k-1)=M_3(K,\sigma_p).
\end{equation}
Let \ $\varphi_2(x, y) = (\mathcal{Q}(x)Q_0^{-1}x, y), \  (x,y)\in \mathbf{B}_{\delta_2}^n$.\\
We are applying the Inverse Mapping Theorem \ref{Inverse} to
$\varphi_2$. So we have to calculate to determine the pair
$(\delta,r)$ (see Remark \ref{remInverse}) and some numbers. First
we have
\[D\varphi_2(x, y)(h, e) = (\mathcal{Q}(x)Q_0^{-1}h+D\mathcal{Q}(x)hQ_0^{-1}x, e).\]
Hence
\[
D\varphi_2(0, 0)(h, e)=(Q_0^{-1}h, e), \textrm{ and } D\varphi_2(0,0)^{-1}(H,e)=(Q_0H,e).\]
Thus
\[\ds\frac 1{\|D\varphi_2(0, 0)^{-1}\|} = \frac{1}{\ds\sup_{\|H\|=1}(\|Q_0h\|^2+1)^{\frac 12}}=\frac 1{\sqrt{\|Q_0\|^2+1}}
=\sqrt{\ds\frac{\sigma_p}{1+\sigma_p} }.\]
 So we get \
 $\delta_3 =\ds\frac{1}2\sqrt{\frac{\sigma_p}{1+\sigma_p} }$.
 Let \
 $r_3=\ds\frac{\delta_3}{2(\bar{K}+1)(1+(\bar{K}+2)^2p(p+1))\sqrt{K}}$. \\
 Applying the Mean Value Theorem and  (\ref{ctDQ}), when $\|(x,y)\|<r_3$, we have
 \[\begin{array}{rcl}
 \|D\varphi_2(x,y)-D\varphi_2(0,0)\|& =&\ds\sup_{\|h\|=1}\|\mathcal{Q}(x)Q_0^{-1}h+ D\mathcal{Q}(x)hQ_0^{-1}x - \mathcal{Q}(0)Q_0^{-1}h\| \\
 &\leq& \|\mathcal{Q}(x)-\mathcal{Q}(0))\|\|Q_0^{-1}\|+\|D\mathcal{Q}(x)\|\|Q_0^{-1}\|\|x\|\\
  & <&  2\|D\mathcal{Q}\|r_3\|Q_0^{-1}\| \\
   &\leq& 2(\bar{K}+1)(1+(\bar{K}+2)^2p(p+1))\sqrt{K}r_3=\delta_3 .
  \end{array}\]
Now  applying the Inverse Mapping Theorem \ref{Inverse} to
$\varphi_2$,
we have \ $\varphi_2^{-1}: \mathbf{B}_{\delta}^n\to \R^n,$\\
 where
 \begin{equation}\label{ctdelta}
 \delta=\ds\frac{\min(\delta_2, r_3)\delta_3}2, \text{ and }
 \|D\varphi_2^{-1}\|\leq\frac 1{\delta_1}=2\sqrt{\frac{1+\sigma_p}{\sigma_p} }.
 \end{equation}
  Let
\[\varphi = \varphi_1\circ \varphi_2^{-1}: \mathbf{B}_{\delta}^n \rightarrow \varphi(\mathbf{B}_{\delta}^n).\]
 Note that we used $L(\varphi_2)=(\bar{K}+1)(1+(\bar{K}+2)^2p(p+1)) >1$, so, by Theorem \ref{Inverse},
$\varphi_2^{-1}(\mathbf{B}_{\delta}^n)\subset
\mathbf{B}_{\delta_1}^n$ (the domain of $\varphi_1$). By the
$C^{k-1}$ coordinate transformation $\varphi$, we have
\[\begin{array}{rcl}
f\circ \varphi(x, y) & = & f_1(\varphi_1(\varphi_2^{-1}(x, y)) + \alpha(y)\\
&=&f_2\circ \varphi_2^{-1}(x, y) + \alpha(y)\\
 & = & \ds\sum_{i=1}^p \pm x_i^2 + \alpha(y).
\end{array}\]
Using (\ref{ctDphi1}) (\ref{ctdelta}) and $\sigma_p\leq K$, we can
easily get the following estimate
\[\begin{array}{lll}
\|D\varphi\|\leq\|D\varphi_1\|\|D\varphi_2^{-1}\|
 &\leq (1+8K^2(K+1)(1+\frac{(K+1)^2}{\sigma_p^2}))2\sqrt{\ds\frac{1+\sigma_p}{\sigma_p}}\\
 & <\ds\frac{32}{\sigma_p^{5/2}}(K+1)^5.
 \end{array}\]
 Moreover, using the Leibnitz Rule, we have
 \[\begin{array}{lll}
 \|\varphi_2\|_{C^{k-1}} &\leq  \|\mathcal{Q}Q_0^{-1}\|_{C^{k-1}} +1\\
 &\leq 2^{k-1}\|\mathcal{Q}\|_{C^{k-1}}\|Q_0^{-1}\|+1\\
 &\leq 2^{k-1}M_3(K,\sigma_p)\sqrt{K} +1=M_4(K,\sigma_p).
 \end{array}
 \]
 From this estimation, (\ref{ctDphi1})
(\ref{ctdelta}) and using Lemmas \ref{lemE}, \ref{lemEI}, we get
\[\begin{array}{lll}
\|\varphi\|_{C^{k-1}}
 &=\|\varphi_1\circ\varphi_2^{-1}\|_{C^{k-1}}\\
 &\leq E(\|\varphi_2^{-1}\|_{C^{k-1}}, \|\varphi_1\|_{C^{k-1}}, k-1) \\
 &\leq E(EI(M_4,2\sqrt{ \tx\frac{1+\sigma_p}{\sigma_p}}, k-1), M_2, k-1)= M(K, \sigma_p,k).
 \end{array}
\]
\textbf{Step 4.} To avoid the complicated formula for $\delta$ we  make some
elementary estimates. Keeping track of the numbers during the proof,
from  \ref{ctdelta}, we have
\[
 \delta =\ds\frac 18{\delta_3} \min(a,b,c),
 \]
 where
 \par
 $a=\ds\frac 1{2K(K+1)(1+\frac{(K+1)^2}{\sigma_p^2} )},$
 \par
 $b=\ds\frac 1{(\frac{K}{\sigma_p}+1)(\frac{K}{\sigma_p}+2)(1+(\frac{K}{\sigma_p}+2)^2p(p+1)},$
 \par
 $c=\ds\frac{\sqrt{\sigma_p}}
      {\sqrt{1+\sigma_p}(\frac{K}{\sigma_p}+1)(1+(\frac{K}{\sigma_p}+2)^2p(p+1))\sqrt{K}}. $
 \\
 Use $\sigma_p\leq K$ to get
 \[ \begin{array} {lll}
 a &=\frac {\sigma_p^2} {K(K+1)(\sigma_p^2+(K+1)^2)} > \frac{\sigma_p^2} {4(K+1)^4},  \\
 b &=\frac{\sigma_p^4} {(K+\sigma_p)(K+2\sigma_p)(\sigma_p^2+(K+2\sigma_p)^2p(p+1))}
    > \frac{\sigma_p^4} {6(K+1)^4(p^2+p+1)},\\
 c &= \frac{\sqrt{\sigma_p}\sigma_p^3} {\sqrt{1+\sigma_p}(K+\sigma_p)(\sigma_p^2+(K+2\sigma_p)^2p(p+1))\sqrt{K} }
 > \frac{\sqrt{\sigma_p}\sigma_p^3}{2(K+1)^4(p^2+p+1)}.
 \end{array}
 \]
 So
 \[ \begin{array} {lll}
 \delta & >\frac 18\sqrt{\frac{\sigma_p}{\sigma_p+1}}\frac{\sigma_p^2}{4(K+1)^4}
 \min(1,\ds\frac{2\sigma_p^2}{3(p^2+p+1)}, \ds\frac{\sigma_p^{3/2}}{2(p^2+p+1)})\\
 & >\ds\frac{\sigma_p^{5/2}}{32(K+1)^{9/2}}
 \min(1,\ds\frac{2\sigma_p^2}{3(p^2+p+1)}, \ds\frac{\sigma_p^{3/2}}{2(p^2+p+1)}).
 \end{array}\]
 We reduce the radius of the domain
of $\varphi$ to the last number to use in the statement of the
theorem.
\end{proof}
\begin{thm}[Morse Lemma] \label{Morse}
 With  the assumptions and notations of Theorem \ref{Split}, when  $p=n$ we have
\[f \circ \varphi(x_1,\cdots, x_n) = f(x_0)+ \sum_{i=1}^n \pm x_i^2.\]
\end{thm}
Applying the quantitative Morse-Sard Theorem (see \cite{Y1} or \cite{Y-C}) and the
Inverse Mapping Theorem, we give here a version for the density
of Morse functions on a ball (c.f. {\cite[Th. 4.1, Th. 6.1]{Y2}}).

\begin{thm}\label{density}
Fix $k\geq$ 3. Let $f_0: \overline{\mathbf{B}}^n \rightarrow
\mathbb{R}$ be a $C^k$-function  with $\|f_0\|_{C^k}\leq K$. Then for any given $\varepsilon > 0$, we
can find $h$ with $\|h\|_{C^k} \leq \varepsilon$ and the positive
functions $\psi_1$, $\psi_2$, $\psi_3$, $d$, $M$, $N$, $\eta$
depending on $K$ and $\varepsilon$, such that $f = f_0 + h$
satisfies the following conditions:\\
{\rm(i) } At each critical point $x_i$ of $f$,  $\sigma_n(Hf(x_i))\geq \psi_1(K, \varepsilon)$.\\
{\rm(ii)} For any two different critical points $x_i$ and $x_j$ of $f$, $\|x_i-x_j\| \geq d(K, \varepsilon)$. \par
         Consequently, the number of  critical points does not exceed $N(K, \varepsilon)$.\\
{\rm(iii)}  For any two different critical points $x_i$ and $x_j$ of $f$, $|f(x_i)-f(x_j)| \geq \psi_2(K, \varepsilon)$.\\
{\rm(iv)}  For  each critical point $x_i$ of $f$, there exists a $C^{k-1}$ coordinate transformation  \par
$\varphi: \mathbf{B}_\delta^n(x_i) \rightarrow \R^n $ such that
\begin{displaymath}
f\circ\varphi^{-1}(y_1,\ldots, y_n) = y_1^2 + \cdots + y_l^2 -
y_{l+1}^2 - \cdots - y_n^2 + const,
\end{displaymath}
where $\delta = \psi_3(K, \varepsilon)$ and $\|\varphi\|_{C^{k-1}} \leq M(K, \varepsilon)$.\\
{\rm(v)}  If  $\|Df(x)\| \leq \eta(K, \varepsilon)$, then $x\in \mathbf{B}_\delta^n(x_i)$, with $x_i$ is a critical point of $f$.
\end{thm}

\begin{proof}
In \cite{L-P}, our proof  of the theorem needs some corrections. Moreover, we can apply the Splitting Lemma
\ref{Split} to get an alternative proof of (iv) in that paper with
more explicit estimations for $\delta$ and $M$. For these reasons, we  make some improvements in detail in this present paper.
\\
\\
\vspace{8pt}\noindent
(i) We are applying the results of Chapter 9 \cite{Y-C} to $Df_0$. For $\gamma>0$, denote $\bar{\gamma}=(\lambda_1,\lambda_2,\cdots,\lambda_n)=(K,K,\cdots,\gamma)$. Then, by definition,  the set of $\bar{\gamma}$-critical points and the set of $\bar{\gamma}$-critical values of $f$ are
\[\begin{array}{lll}
\Sigma(Df_0,\bar{\gamma}, \overline{\mathbf{B}}^n)
&=\{x\in\overline{\mathbf{B}}^n:\sigma_i(D(Df_0)(x))\leq \lambda_i, i=1,\cdots,n\}\\
&=\{x\in\overline{\mathbf{B}}^n:\sigma_n(Hf_0(x))\leq \gamma\}, \textrm{and}\\
\Delta(Df_0,\bar{\gamma}, \overline{\mathbf{B}}^n)
&= f(\Sigma(Df_0,\bar{\gamma}, \overline{\mathbf{B}}^n)).
\end{array}
\]
For a  relatively compact subset $A $ of $\R^n$, and $r > 0$, denoted by $M(r, A)$ the
minimal number of balls of radius $r$ in $\R^n$, covering $A$. \\
Let $ \varepsilon >0$. Applying Theorem 9.6 of \cite{Y-C}, when $0<r<\varepsilon$,
\[\begin{array}{lll}
& M(r, \Delta(Df_0, \bar{\gamma}, \overline{\mathbf{B}}^n) \cap\mathbf{B}_\varepsilon^n)
\leq c\left( \frac{R_k(f_0)}{r}\right)^\frac{n}{k}
 \sum_{i = 0}^n \min\left(\lambda_0\cdots\lambda_i\frac 1{r^i}\left(\frac{r}{R_k(f_0)}\right)^{\frac ik}, \left(\frac{\varepsilon}r \right)^i\right)  \\
 &\leq  c\left(\frac{R_k(f_0)}{r}\right)^\frac{n}{k}
 \left[\sum_{i = 0}^{n-1}\min\left(K^i\frac 1{r^i}\left(\frac{r}{R_k(f_0)}\right)^{\frac ik}, \left(\frac{\varepsilon}r \right)^i\right) +
 \min\left(K^{n-1}\gamma \frac 1{r^n}\left(\frac{r}{R_k(f_0)}\right)^\frac n{k}, \left(\frac{\varepsilon}r \right)^n\right)\right],
\end{array}
\]
where $\lambda_0=1, c=c(n,k)$ and $R_k(f_0) = \frac{K}{(k-1)!}$. If  $0<r<1$ and $r<\frac{R_k(f_0)\varepsilon^k}{K^k}$, then by taking the $\min$ and simplifying
the right-hand side we get
\[
 M(r, \Delta(Df_0, \bar{\gamma}, \overline{\mathbf{B}}^n) \cap\mathbf{B}_\varepsilon^n)
 \leq
c\left(\sum_{i=0}^{n-1}K^i(R_k(f_0)^{\frac 1k})^{n-i}\frac 1{r^{n-1+\frac 1k}} +K^n\frac{\gamma}{r^n}\right).
\]
When $\gamma=r^{1-\frac 1k}$, we have
\begin{equation}\label{entropy}
M(r, \Delta(Df_0, \bar{\gamma}, \overline{\mathbf{B}}^n) \cap\mathbf{B}_\varepsilon^n)
 \leq
c\sum_{i=0}^{n}K^i(R_k(f_0)^{\frac 1k})^{n-i}\frac 1{r^{n-1+\frac 1k}}.
\end{equation}
Note that, by the definition of $M(r,A)$, it is easy to see that
$M(2r, A_r)\leq M(r, A)$, where $A_r$ denotes the $r$-neighborhood
of subset $A$ of $\R^n$. Therefore, if
\begin{equation}\label{dodo}
 M(r, \Delta(Df_0, \bar{\gamma}, \overline{\mathbf{B}}^n)\cap\mathbf{B}_\varepsilon^n)m(\B_{2r})
 <m(\B^n_{\varepsilon}),
 \end{equation}
 where $m(A)$ denotes the Lebesgue measure of $A$,
then there exists $v_0\in \B^n_{\varepsilon}$, such that $v_0$ is
not contained in a union of balls of radii $<2r$ that covers the
$r$-neighborhood of
$\Delta(Df_0,\bar{\gamma},\overline{\mathbf{B}}^n)\cap\mathbf{B}_\varepsilon^n$,
and hence
 $\B^n_r(v_0)\cap
 \Delta(Df_0,\bar{\gamma},\overline{\mathbf{B}}^n)=\emptyset$.\\
 We want to find $r$, $0<r<\min(\varepsilon, 1,\frac{R_k(f_0)\varepsilon^k}{K^k})$ satisfying \ref{dodo}. Combining \ref{entropy} and  \ref{dodo}, we look for  $r$ satisfying
 \[c\sum_{i=0}^{n}K^i(R_k(f_0)^{\frac 1k})^{n-i}\frac 1{r^{n-1+\frac 1k}}<\frac{\varepsilon^n}{(2r)^n},\]
 or
 \[r<\left(\frac{\varepsilon^n}{2^nc\sum_{i=0}^{n}K^i(R_k(f_0)^{\frac 1k})^{n-i}}\right)^{\frac{k}{k-1}}
 .\]
Taking
\[
\begin{array}{lll}
r(K, \varepsilon)
 &=\frac 12\min\left( \varepsilon,1, \frac{R_k(f_0)\varepsilon^k}{K^k},
 \left(\frac{\varepsilon^n}{2^nc\sum_{i=0}^nK^i(R_k(f_0)^{\frac 1k})^{n-i} }\right)^{\frac{k}{k-1}} \right),\textrm{ and }\\
\gamma(K, \varepsilon) &= r(K,\varepsilon)^{1-\frac 1k},
\end{array}\]
 we get \ref{dodo}.
Then we can choose  $v\in \B^n_{\varepsilon}$, such that $\B^n_{\frac{r(K,\varepsilon)}2}(v)\subset \B^n_{\varepsilon}$ and every $v'$ in $\B^n_{\frac{r(K,\varepsilon)}2}(v)$
 is a $\gamma(K, \varepsilon) $-regular value of $Df_0$.
\\
Now, let $l: \R^n \rightarrow \R$ be a linear mapping with $Dl = -v$
and $f_1=f_0+l$. Then $\|l\|_{C^k} \leq \varepsilon -
\frac{r(K,\varepsilon)}2$,
$Df_1=Df_0-v$, and $Hf_1=Hf_0=D(Df_0)$. So each $v'\in \B^n_{\frac{r(K,\varepsilon)}2}(v)$ is a $\gamma(K,
\varepsilon)$-regular value of $Df_1$. In particular,  at each critical point
$x_i$ of $f_1$, we have

\begin{equation}\label{ctsigma}
\|Hf_1(x_i)\| \geq \sigma_n(Hf_1(x_i))\geq \gamma(K, \varepsilon).
\end{equation}
In other words,  the smallest absolute value of the eigenvalues of
the Hessian of $f_1$ at its critical points is at least $\psi_1(K,
\varepsilon) = \gamma(K, \varepsilon).$

\vspace{8pt}\noindent
 (ii) We are applying the Inverse Mapping
Theorem \ref{Inverse} to $Df_1: \overline{\mathbf{B}}^n \rightarrow \R^n$ at the
critical points of $f_1$. Let $x_i$ be a critical point of $f_1$. By
(\ref{ctsigma}) we have
\[\frac{1}{2}\frac{1}{\|Hf_1(x_i)^{-1}\|}\geq \frac{1}{2}\gamma(K, \varepsilon).\]
Choose $\delta' = \frac{1}{2}\gamma(K, \varepsilon)$, and
$r'=\ds\frac{\delta'}K$. Applying the Mean value theorem, when
$\|x-x_i\|<r$, we have
\[\|D(Df_1)(x) - D(Df_1)(x_i)\| = \|D(Df_0)(x) - D(Df_0)(x_i)\| \leq K \|x -x_i\|<Kr'<\delta'.\]
Thus, by  Theorem \ref{Inverse}, $Df_1$ is
invertible on $\mathbf{B}_{\frac{r'\delta'}{2K}}^n\left( x_i\right) =
\mathbf{B}_{\frac{\gamma^2(K, \varepsilon)}{8K^2}}^n\left( x_i
\right)$.  Hence, $Df_1^{-1}(0)\cap \mathbf{B}_{\frac{\gamma^2(K,
\varepsilon)}{8K^2}}^n\left( x_i \right)$ has only one point, i.e.
$x_i$ is the unique critical point of $f_1$ in the ball
$\ds\mathbf{B}_{\frac{\gamma^2(K, \varepsilon)}{8K^2}}^n\left( x_i
\right)$. So the distance between any two different critical points
$x_i,x_j$ of $f_1$ can be estimated from below by
\[d(x_i, x_j) \geq d(K,\varepsilon) = \frac{1}{4}\frac{\gamma^2(K,\varepsilon)}{K^2}.\]
Therefore, the number of critical points of $f_1$  does not exceed
\[N(K,\varepsilon) = M\left(\frac{1}{4}\frac{\gamma^2(K,
\varepsilon)}{K^2},\mathbf{B}^n\right) \leq \frac{1}{8}\frac{\gamma^2(K,
\varepsilon)}{K^2}n^2.\]

\vspace{8pt}\noindent
 (iii) Suppose that the critical
points of $f_1$ are $x_1,\cdots, x_N, N \leq N(K, \varepsilon)$, and
the critical values of $f_1$ are ordered increasingly
$$f_1(x_1) \leq f_1(x_2) \leq \ldots \leq f_1(x_{N}).$$
Let $g:\R\to\R$ be a $C^k$ function satisfying the following
conditions
\[\begin{array}{ll}
g(t) = \left\lbrace \begin{array}{rl}
1\ , &\textrm{if } |t|< \frac{d(K,\varepsilon)}4,\\
0\ ,&\textrm{if }|t|>\frac{d(K,\varepsilon)}2,
\end{array}\right.\\
0<g(t)<1, \textrm{ if } \frac{d(K,\varepsilon)}4\leq |t|\leq
\frac{d(K,\varepsilon)}2.
\end{array}
\]
For each $i$, let \ $\lambda_i:\R^n \rightarrow [0, 1] $ be defined by\
$\lambda_i(x)=g(\|x-x_i\|)$,  and  $C_1=\|\lambda_i\|_{C^k}$. Put
$\eta_1=\min(r(K,\varepsilon),
\frac{\gamma^2(K,\varepsilon)}{8(K+\varepsilon)})$. (The second
parameter of $\min$ will be used in  (v)). Let
\[\lambda: \R^n \rightarrow \R,~~~ \ds \lambda(x) = \sum_{i=1}^{N}c_i\lambda_i(x), \
\textrm{where} \ c_i = (i-1)\frac{\eta_1}{4NC_1}.\] Then
$\|\lambda\|_{C^k}={\ds\max_{1\leq i\leq N}}c_i\|\lambda_i\|_{C^k}<\frac{\eta_1}4\leq
\frac{r(K,\varepsilon)}4$. Now consider the approximation of $f_0$:
\[f= f_1 + \lambda=f_0+h, \textrm{ where } h=l+\lambda.\]
We have \ $\|h\|_{C^k}\leq \|v\|+\|\lambda\|_{C^k}< (\varepsilon -
\frac{r(K,\varepsilon)}2)+\frac{r(K,\varepsilon)}4<\varepsilon.$\\
Since \ $Df(x)=0$ iff $Df_1(x)=-D\lambda(x)$, by the
definition of $\lambda$, this equality only happens when $x\in
\B^n_{\frac{d(K,\varepsilon)}2}(x_i)$ for some $i$. But $Df_1$ is
injective on $\B^n_{\frac {d(K,\varepsilon)}2}(x_i)$ and
$D\lambda(x)=D\lambda(x')$, when $\|x-x_i\|=\|x'-x_i\|$, so $x$ must
be equal to $x_i$, and then
$Hf(x_i)=Hf_1(x_i)+H\lambda(x_i)=Hf_0(x_i)$. Thus $f$ is a Morse
function having the same critical points as $f_1$, and
$\sigma_n(Hf(x_i))\geq \gamma(K,\varepsilon)$ for every critical point
$x_i$. \\
Moreover, for any pair of distinct critical points $x_i,
x_j$ of $f$, we have
\[
 |{f}(x_i) - {f}(x_j)| = |f_1(x_i) + c_i - f_1(x_j) - c_j|
\geq \psi_2(K,\varepsilon)=\frac{\eta_1(K,\varepsilon)}{4NC_1}.
\]
We showed that $f$ satisfies (i), (ii) and (iii).\\
\vspace*{8pt}\noindent
(iv) Applying the Splitting Lemma \ref{Split}
to $f$ at each of its critical points we get
$\delta=\delta(K,\varepsilon)$ and $M(K,\varepsilon)$  satisfying
(iv).\\
\vspace*{8pt}\noindent
 (v) First, consider $f_1=f_0+l$.  If $\|Df_1(x)\|\leq \frac 12\eta_1=\frac 12\min(r(K,\varepsilon),
\frac{\gamma^2(K,\varepsilon)}{8(K+\varepsilon)})$, then
$\|Df_0(x)-v\|\leq \frac{r(K,\varepsilon)}2$, and hence, from (i) we
have $\sigma_n(Hf_1(x))=\sigma_n(Hf_0(x))\geq
\gamma(K,\varepsilon)$. According to the Inverse Mapping Theorem
\ref{Inverse}, $Df_1$ is invertible on a ball centered at $x$ with
radius $\rho_1=\frac{(\gamma/2)^2}{2(K+\varepsilon)^2}$, and the
image contains the ball
 centered at $D{f_1}(x)$ with radius $\rho_2=\rho_1 (K+\varepsilon) > \eta_1$,
 and hence this ball contains $0$.
Therefore, if $\|Df_1(x)\|\leq \frac 12\eta_1=\frac
12\min(r(K,\varepsilon),\rho_2)$, then there exists a
critical point $x_i$ of $f_1$ such that $x\in\B^n_{\rho_1}(x_i)$.\\
Now consider  \ $f=f_1+\lambda$.  If
$\|Df(x)\|<\frac 14 \eta_1$, then $\|Df_1(x)\|\leq \|Df(x)\|+\|D\lambda(x)\| \leq
\frac 14\eta_1+\frac 14\eta_1=\frac 12\eta_1$, and hence
$x\in\B^n_{\rho_1}(x_i)$. Note that $Df(x_i)=Df_1(x_i)=0$.\\
Therefore, to get (v) we take
\[\eta(K,\varepsilon)=\frac 14\min(r(K,\varepsilon),
\frac{\gamma^2(K,\varepsilon)} {8(K+\varepsilon)}) , \textrm{ and }
\psi_3(K,\varepsilon)=\min(\delta(K,\varepsilon),\frac{\gamma^2}{8(K+\varepsilon)^2}).
\]

\end{proof}

 Applying the Inverse Mapping Theorem, we get a quantitative version
for the openness of Morse functions on a ball as follows.

\begin{thm}[]\label{Openess}
Let $f: \overline{\mathbf{B}}^n \rightarrow \mathbb{R}$ be a $C^k$
function, $k\geq 2$,  with $\|f\|_{C^k}\leq K$. Suppose that
 $f$ is a Morse function with the critical locus $\Sigma(f)=\{x_1,\cdots, x_p\}$ contained in $\mathbf{B}^n$
  and  has distinct critical values. Let
\[\begin{array}{ll}
\gamma=\min\{\sigma_n(Hf(x)): x\in\Sigma(f)\}, \\
 d=\min\{|f(x_i)-f(x_j)|: i\neq j\textrm{ and } i,j=1,\cdots, p\},\\
 \rho =\min(\frac{\gamma^2}{128K^2}, \frac
 d{8K},d(\Sigma(f),\partial\B^n)),  \textrm{and } \\
\eta=\inf \{\|Df(x)\|: x\in \overline{\B}^n, d(x,\Sigma(f))\geq \rho\}.
 \end{array}\]
  Let $\varepsilon=\min(\frac{\eta}2, \frac{\gamma^2}{64K},\frac d4)$.
Then
 for every $C^k$ function
$\bar{f}:\overline{\mathbf{B}}^n \rightarrow \mathbb{R}$, with
$\|\bar{f} - f\|_{C^k}<\varepsilon$,
$\bar{f}$ satisfies the followings:\\
 \rm{(i)} If $x\in \B^n$ and $\|D\bar{f}(x)\|<\frac{\eta}2$, then $\sigma_n(H\bar{f}(x))\geq \frac{\gamma}2$.
 In particular, $\bar{f}$ is a Morse function.\\
 \rm{(ii)} $\Sigma(\bar{f})=\{\bar{x}_1,\cdots, \bar{x}_p\}\subset \B^n$, and
 $\|\bar{x_i}-x_i\|<\rho $, for $i=1,\cdots,p$.\\
 \rm{(iii)} $\sigma_n(H\bar{f}(\bar{x}_i))\geq \frac{\gamma}2 $, for $i=1,\cdots,p$.\\
 \rm{(iv)} $\min\{|\bar{f}(\bar{x}_i)-\bar{f}(\bar{x}_j)|: i\neq j \textrm{ and } i,j=1,\cdots, p\}\geq \frac d2.$
\end{thm}
\begin{proof} First note that $\eta>0$, since $\Sigma(f)\cap \partial B^n=\emptyset . $\\
Let $\bar{f}:\overline{\mathbf{B}}^n \rightarrow \mathbb{R}$ be a
$C^k$ function with $\|\bar{f} - f\|_{C^k}<\varepsilon$. \\
Let $x\in\B^n$, such that $\|D\bar{f}(x)\|<\frac{\eta}2$. Then the
definition of $\varepsilon$ implies
\[ \|Df(x)\|<\frac{\eta}2+\varepsilon \leq \eta,\]
 and hence $d(x,\Sigma(f))<\rho$, i.e.  $\|x-x_i\|<\rho.$ for some $i\in\{1,\cdots,p\}$.\\
Moreover, by $\|\bar{f}\|_{C^k}<K+\varepsilon$ and the definition of
$\rho$ and $\varepsilon$, applying the Mean value theorem, we get
\[\begin{array}{lll}
\|H\bar{f}(x)-Hf(x_i)\| &\leq
\|H\bar{f}(x)-H\bar{f}(x_i)\|+\|H\bar{f}(x_i)-Hf(x_i)\|\\
 & \leq (K+\varepsilon)\|x-x_i\|+ \varepsilon \\
 &< (K+\varepsilon)\rho+\varepsilon \\
 & \leq (1+\rho)\varepsilon +K\rho <\frac{\gamma}2.
\end{array}\]
Therefore $\sigma_n(H\bar{f}(x))\geq \frac{\gamma}2.$ This proves
(i) and (iii).\\
For $x_i\in\Sigma(f)$, we have $\|D\bar{f}(x_i)\|<\varepsilon\leq
\frac{\eta}2$. From (i) and the Inverse Mapping Theorem
\ref{Inverse}, $D\bar{f}$ is invertible on a ball centered at $x_i$
with radius $\frac{(\gamma/4)^2}{2(K+\varepsilon)^2}\geq
\frac{\gamma^2}{128K^2}\geq \rho$, and the image contains a ball
 centered at $D\bar{f}(x_i)$ with radius $\frac{(\gamma/4)^2}{2(K+\varepsilon)}>\frac{\gamma^2}{64K}\geq
 \varepsilon>\|D\bar{f}(x_i)\|$. From these facts and (i), there
 exists $\bar{x}\in \Sigma(\bar{f})$ such that
 $\|\bar{x}-x_i\|<\rho$.
 Note that, by the definition of $\rho$, all critical points of $\bar{f}$ are contained in $\B^n$. \\
 Moreover, for any two distinct critical points $\bar{x},\bar{y}\in
 \Sigma(\bar{f})$, $\|\bar{x}-\bar{y}\|>2\rho$. Hence for all $x_i\in \Sigma(f)$ there
 exists only one $\bar{x}_i\in \Sigma(\bar{f})$ such that
 $\|\bar{x}_i-x_i\|<\rho$, (ii) follows.\\
 To prove (iv), adding a constant to $\bar{f}$,  we can assume $\bar{f}(0)=f(0)$.
 Then, by the Mean value theorem,  $|\bar{f}(x)-f(x)|<\varepsilon$, for
 all $x\in\B^n$. So for $i=1,\cdots, p$, we have
 \[\begin{array}{lll}
 |\bar{f}(\bar{x}_i)-f(x_i)|
 &\leq |\bar{f}(\bar{x}_i)-f(\bar{x}_i)|+|f(\bar{x}_i)-f(x_i)|\\
 &<\varepsilon+K\|\bar{x}_i-x_i\|\leq \frac d4+K\rho <\frac d2.
 \end{array}
 \]
By the triangle inequality, for any two distinct critical points
$\bar{x}_i,\bar{x}_j\in \Sigma(\bar{f})$, we have
\[|\bar{f}(\bar{x}_i)-\bar{f}(\bar{x}_j)|\geq \frac d2.\]
\end{proof}
\noindent
{\it Acknowledgement.} This research is supported by Vietnamese National Foundation for Science and Technology
Development 2010 - 2012.

\end{document}